\documentclass[11pt]{article}

\usepackage{fullpage}
\usepackage{amsmath, amsthm, amsfonts, amssymb, amstext, mathrsfs, enumerate}
\usepackage{graphicx, ragged2e, lscape, framed, xcolor}
\usepackage{subfiles}

\theoremstyle{plain}
\newtheorem{theorem}{Theorem}[section]
\newtheorem{lemma}[theorem]{Lemma}

\newtheorem{conjecture}[theorem]{Conjecture}

\newtheorem{claim}{Claim}[section]

\numberwithin{equation}{section}
\allowdisplaybreaks

\usepackage[pagebackref]{hyperref}
\hypersetup{
	colorlinks=true,
    urlcolor=purple,
	linkcolor=purple,
    citecolor=purple,
}

\DeclareMathOperator{\tr}{tr}
\def\x{\mbox{\boldmath $x$}}

\title{\textbf{Energy and independence number}}
\author{Hitesh Kumar \and Shivaramakrishna Pragada}
\date{}

\begin{document}
\maketitle
\begin{abstract} 
For a graph $G$ of order $n$, with adjacency eigenvalues $\lambda_1(G) \geq  \cdots \geq \lambda_n(G)$, the \emph{energy} of $G$ is defined to be 
\[\mathcal{E}(G)=\sum_{i=1}^{n} |\lambda_i(G)|.\]
A well-known conjecture from the 1980s by Fajtlowicz states that for any graph $G$, 
\[\mathcal{E}(G) \ge 2\left(n-\alpha(G)\right),\]
where $\alpha(G)$ denotes the independence number. We prove this conjecture.
\end{abstract}

\noindent
\textbf{Keywords:} Graph Energy, Independence number, Semidefinite Programming

\noindent
\textbf{MSC2020:} 05C50, 15A18, 90C22.

\section{Introduction}

Let $G = (V(G), E(G))$ be a finite simple graph of order $n = |V(G)|$ and size $m = |E(G)|$. If two vertices $u, v\in V(G)$ are adjacent in $G$, we write $u\sim v$, otherwise we write $u\nsim v$. Let $\alpha(G)$ denote the independence number of $G$. For a vertex $v\in V(G)$, the \emph{open-neighbourhood} (resp. \emph{closed-neighbourhood}) of $v$ is the set $N(v)=\{u: u\sim v\}$ (resp. $N[v] = N(v)\cup \{v\}$). For a subset $U \subseteq V(G)$, let $G-U$ denote the subgraph obtained after deleting the vertices in $U$ from $G$. 

The \textit{adjacency matrix} of $G$ is an $n\times n$ matrix $A(G) = [a_{uv}]$, where $a_{uv} = 1$ if $u\sim v$ and $a_{uv} = 0$, otherwise. The \emph{eigenvalues} of $G$ are the eigenvalues of $A(G)$. Since $A(G)$ is a real symmetric matrix, all eigenvalues of $A(G)$ are real which we enumerate as 
\[\lambda_1(G) \geq  \cdots \geq \lambda_n(G).\] 
The \emph{energy} of $G$, denoted $\mathcal{E}(G)$, is defined to be 
\[\mathcal{E}(G):=\sum_{i=1}^{n} |\lambda_i(G)| = 2\sum_{\lambda_i > 0} \lambda_i(G) = -2\sum_{\lambda_i < 0} \lambda_i(G),\]
where the last two equalities hold because $\tr(A(G)) = 0$. The energy of a graph is an extensively studied parameter and finds applications in many areas, including mathematical chemistry; see the surveys by Gutman \cite{Gutman_2001, Gutman_2017_survey}, Nikiforov \cite{Nikiforov_2016_GraphEnergy}, and the citations in and of \cite{Abiad_2025_SDPlenses, Akbari_Seidel_2020, Einollahzadeh_2024}.

Using Graffiti in the 1980s, Fajtlowicz \cite{DeLaVina_2005, Fajtlowicz_1980} proposed the following conjecture concerning the energy and the independence number of graphs, which appears in Written on the Wall (cf. \cite{Aouchiche_Hansen_2010}) and was recently popularized by Liu and Ning \cite{Liu_Ning_OpenProblems_2023} in their survey of unsolved problems in spectral graph theory. 

\begin{conjecture}[{\cite[Conjecture \#543, Table 6]{Aouchiche_Hansen_2010}} cf. \cite{Liu_Ning_OpenProblems_2023}]\label{conj:energy}
    For any graph $G$ of order $n$, we have 
    \[ \mathcal{E}(G)\ge 2(n-\alpha(G)).\]
\end{conjecture}

Despite its simplicity, the conjecture has resisted a proof for a long time. Recall that for any graph $G$ of order $n$,
\[ \alpha(G) + \tau(G) = n,\]
where $\tau(G)$ is the \emph{vertex cover number} of $G$. Reformulated in terms of $\tau(G)$, Conjecture \ref{conj:energy} reads: for any graph $G$, we have 
\begin{equation}\label{eq:tau}
   \mathcal{E}(G)\ge 2\tau(G). 
\end{equation}
Apparently unaware of the original appearanc of Conjecture \ref{conj:energy}, Akbari, K{\"u}{\c{c}}{\"u}k{\c{c}}if{\c{c}}i, Saveh, and Yaz{\i}c{\i} \cite{Akbari_2025_Vertex_cover} posed \eqref{eq:tau} as a conjecture in their recent paper. They were motivated by a result of Wang and Ma \cite{Wang_Ma_2017}, who established 
\begin{equation}\label{eq:Wang_Ma}
   \mathcal{E}(G)\ge 2(\tau(G)-c(G)), 
\end{equation}
where $c(G)$ is the number of odd cycles in $G$. Akbari et al. \cite{Akbari_2025_Vertex_cover} strengthened \eqref{eq:Wang_Ma} and also verified Conjecture \ref{conj:energy} for several graph families, such as perfect graphs. Samanta \cite{Samanta_2025} established Conjecture \ref{conj:energy} for graph families such as split graphs, cycle-clique graphs, etc. Recently, Abiad, Coutinho, Juliano, and Reijnders \cite{Abiad_2025_SDPlenses} approached the problem via semidefinite programming and made significant progress. Indeed, they proved that
\begin{equation}\label{eq:fractional_chromatic}
    \mathcal{E}(G)\ge 2(n- \chi_f(\overline{G}))
\end{equation}
where $\chi_f(\overline{G})$ is the \emph{fractional chromatic number} of the complement $\overline{G}$ of $G$. 

In this paper, we settle Conjecture \ref{conj:energy}.

\begin{theorem}\label{thm:main}
    For any graph $G$ of order $n$, 
    \[ \mathcal{E}(G)\ge 2(n-\alpha(G)).\]
\end{theorem}

We prove the above result in Section \ref{sec:main_proof} using an inductive argument on the number of vertices. This is enabled by a \emph{neighbourhood deletion inequality} (Lemma \ref{lemma:deletion}) that we believe is of independent interest.

The well-known inertia bound states that 
\[\alpha(G)\le \min\{n-n^+(G),\, n-n^-(G)\},\]
where $n^+(G)$ (resp. $n^-(G)$) denote the number of positive (resp. negative) eigenvalues of $G$. 
As a corollary, we obtain 
\[ \mathcal{E}(G)\ge 2\max\{n^+(G),\, n^-(G)\} \]
which resolves \cite[Conjectures \#20, \#21, Table 6]{Aouchiche_Hansen_2010} (cf. \cite{Liu_Ning_OpenProblems_2023}).

Moreover, 
\[\alpha(G)\le \vartheta^-(G) \le\vartheta(G) \le \xi_f(\overline{G}) \le \chi_f(\overline{G})\] 
where $\vartheta^-(G)$ is the \emph{Schrijver theta function}, $\vartheta(G)$ is the \emph{Lov\'{a}sz theta function}, and $\xi_f(\overline{G})$ is the \emph{projective rank}. Theorem \ref{thm:main} recovers \eqref{eq:fractional_chromatic} and other results from \cite{Abiad_2025_SDPlenses}, and also resolves \cite[Conjecture 6.1]{Abiad_2025_SDPlenses}.

\section{Neighbourhood deletion inequality}
We recall the following semidefinite formulation of graph energy from Abiad et al. \cite{Abiad_2025_SDPlenses}.

\begin{lemma}[\cite{Abiad_2025_SDPlenses}]\label{lemma:SDP_formulation}
For every graph $G$,
\begin{equation}\label{eq:graph-sdp}
    \mathcal{E}(G)
    =2\min\bigl\{\tr(M): M\succeq 0, \ M-A(G)\succeq 0 \bigr\}.
\end{equation}
\end{lemma}

The following key lemma allows us to prove our main theorem through induction. 

\begin{lemma}[Neighbourhood deletion inequality]\label{lemma:deletion} For any graph $G$ of order $n$ and size $m$, 
\begin{equation}\label{eq:deletion_inequality}
     4m + \sum_{v\in V(G)}\mathcal{E}(G-N[v])\le n\,\mathcal{E}(G).
\end{equation}
\end{lemma}

\begin{proof} Suppose that $G = G_1 \sqcup G_2$ where $G_1$ and $G_2$ are connected graphs. If $\eqref{eq:deletion_inequality}$ holds for $G_1$ and $G_2$, then see that 
\begin{align*}
4m + \sum_{v\in V(G)}\mathcal{E}(G-N[v]) & =  \left(4|E(G_1)|+\sum_{v\in V(G_1)}\bigl(\mathcal{E}(G_1-N_{G_1}[v]) + \mathcal{E}(G_2)\bigr)\right) \\
& \quad + \left(4|E(G_2)|  + \sum_{v\in V(G_2)}\bigl(\mathcal{E}(G_2-N_{G_2}[v]) + \mathcal{E}(G_1)\bigr)\right)\\
& \le |V(G_1)|\, \bigl(\mathcal{E}(G_1) + \mathcal{E}(G_2)\bigr)  + |V(G_2)|\,\bigl(\mathcal{E}(G_1) + \mathcal{E}(G_2)\bigr) \\
& = n\,\mathcal{E}(G).
\end{align*}
Thus, it suffices to prove \eqref{eq:deletion_inequality} for connected graphs.

So let $G$ be a connected graph of order $n\ge 2$. Throughout this proof, let $A= A(G)$. Write 
\[ A = P - Q,\]
where $P$ and $Q$ are the PSD matrices in the spectral decomposition of $A$ (see \cite[Theorem 4.1.5]{Horn_Johnson_2013}) and they satisfy $PQ = QP = 0$. Let 
\[ B: = P + Q.\]
Then $B^2 = A^2$. Since $A$ has zero diagonal, we see that
\begin{equation}\label{eq:B_entry}
    P_{vv} = Q_{vv} = \frac{1}{2}B_{vv}
\end{equation}
for all $v\in V(G)$. Moreover, we have 
\begin{equation}\label{eq:trace_P_energy}
    \sum_{v\in V(G)} P_{vv} = \tr(P) = \frac{1}{2}\mathcal{E}(G).
\end{equation}
Observe that $P_{vv} > 0$ for all $v\in V(G)$. Indeed, if $P_{vv} = 0$ for some $v$, then $B_{vv} = 0$. Since $B\succeq 0$, this forces $B_{uv} = 0$ for all $u\in V(G)$. But then $\deg(v) = (B^2)_{vv} = 0$, which contradicts the fact that $G$ is connected.

For $v\in V(G)$, let $S(v):=V(G)\backslash N[v]$ and let $\x_v$ be the column of $P$ corresponding to vertex $v$ restricted to $S(v)$. Define 
\[P_{v}:= P[S(v)] - \frac{\x_v\x_v^\top}{P_{vv}},\]
where $P[S(v)]$ denotes the principal submatrix of $P$ indexed by vertices in $S(v)$. 

\begin{claim}\label{claim:vertex_deleted_trace} For any vertex $v\in V(G)$,
   \[\mathcal{E}(G-N[v])\le 2\tr(P_v).\] 
\end{claim}

\begin{proof}
First, observe that $P_v$ is the Schur complement of the $1\times 1$ block $P_{vv}$ of the matrix $P[S(v)\cup \{v\}]$ (i.e., the principal submatrix of $P$ indexed by the vertices in $S(v)\cup \{v\}$). Since $P_{vv}>0$, by the generalized Schur-complement criterion for positive semidefiniteness (see \cite[Theorem 1.20]{Zhang_2005_Schur}), we conclude that $P_v\succeq 0$. 

Second, since $P_{uv} = Q_{uv}$ for all $u\in S(v)$, we see that $\x_v$ is also the restriction to $S(v)$ of the column of $Q$ corresponding to vertex $v$. Let 
\[ Q_v: = Q[S(v)] - \frac{\x_v\x_v^\top}{Q_{vv}}. \]
Then, similar to $P_v$, we see that $Q_v\succeq 0$. Now, since
\[ P_v - A(G-N[v]) = Q_v,\]
we see that $P_v-A(G-N[v])\succeq 0$. Thus, $P_v$ is feasible for the SDP formulation for $\mathcal{E}(G-N[v])$ given in Lemma~\ref{lemma:SDP_formulation} and the claim follows.
\end{proof}

\begin{claim}\label{claim:B_edge_inequality} For any edge $uv\in E(G)$, 
\[ \frac{(B_{uu}-1)^2-(B_{uv})^2}{B_{uu}} + \frac{(B_{vv}-1)^2-(B_{uv})^2}{B_{vv}}\ge 0.\]
\end{claim}

\begin{proof} For the edge $uv$, let 
\[x := B_{uu},\quad y := B_{vv}\quad \text{and}\quad z:=B_{uv}.\] 
Since $2P = B + A$ and $2Q = B-A$, the $2\times 2$ principal submatrices of $2P$ and $2Q$ corresponding to edge $uv$ satisfy
\[ 
\begin{pmatrix}
x & z+1\\
z+1& y
\end{pmatrix} \succeq 0 \quad \text{and}\quad 
\begin{pmatrix}
x & z-1\\
z-1 & y
\end{pmatrix} \succeq 0.\]
Their determinants give
\[xy\ge (z+1)^2 \quad \text{and}\quad xy\ge (z-1)^2,\]
which implies $ \sqrt{xy}\ge 1 + |z|$ and consequently
\begin{equation}\label{eq:sqrt_xy}
  (\sqrt{xy}-1)^2 - z^2\ge 0.
\end{equation}
Thus, 
\begin{align*}
    \frac{(x-1)^2 - z^2}{x} + \frac{(y-1)^2-z^2}{y}& = \frac{x+y}{xy}\left(xy +1 -z^2\right) - 4\\
    & \ge \frac{2}{\sqrt{xy}}\left(xy +1 -z^2\right) - 4 \quad (\text{AM-GM})\\
    & = \frac{2}{\sqrt{xy}}\left((\sqrt{xy} - 1)^2 -z^2\right)\\
    & \ge 0 \quad (\text{by \eqref{eq:sqrt_xy}}).
\end{align*}
The claim holds.
\end{proof}

\begin{claim}\label{claim:trace_sum} We have 
\[ 2\sum_{v\in V(G)}\tr(P_v)\le n\, \mathcal{E}(G) - 4m.\]
\end{claim}

\begin{proof}
Observe that 
\begin{equation}\label{eq:trace_P_1}
    \tr(P_v) = \sum_{u\in S(v)} P_{uu} - \frac{1}{P_{vv}}\sum_{u\in S(v)}(P_{uv})^2.
\end{equation}
For each $u\in V(G)$, there are $n-1-\deg(u)$ vertices $v$ such that $u\in S(v)$. Thus, summing \eqref{eq:trace_P_1} over $v$, we see that 
\begin{align*}
    2\sum_{v\in V(G)} \tr(P_v) & =  2\sum_{u\in V(G)} (n-1-\deg(u))P_{uu} - \sum_{v\in V(G)}\frac{2}{P_{vv}}\sum_{u\in S(v)}(P_{uv})^2\\
    & = 2n\left(\sum_{v\in V(G)}P_{vv}\right) - 2\sum_{v\in V(G)} (\deg(v)+1)P_{vv} - \sum_{v\in V(G)}\frac{2}{P_{vv}}\sum_{u\in S(v)}(P_{uv})^2\\
    & = n\,\mathcal{E}(G) - \sum_{v\in V(G)} (\deg(v)+1)B_{vv} - \sum_{v\in V(G)}\frac{1}{B_{vv}}\sum_{u\in S(v)}(B_{uv})^2,
\end{align*}
where the last equality holds by \eqref{eq:trace_P_energy} and \eqref{eq:B_entry} and the fact that $2P_{uv} = B_{uv}$ whenever $u\in S(v)$. So to prove the claim, we only need to show that 
\[\sum_{v\in V(G)} (\deg(v)+1)B_{vv} + \sum_{v\in V(G)}\frac{1}{B_{vv}}\sum_{u\in S(v)}(B_{uv})^2\ge 4m,\]
which is equivalent to
\[\sum_{v\in V(G)} (\deg(v)+1)B_{vv} + \sum_{v\in V(G)}\frac{1}{B_{vv}}\sum_{u\in S(v)}(B_{uv})^2 - 2\sum_{v\in V(G)}\deg(v) \ge 0,\]
which, in turn, is equivalent to 
\begin{equation}\label{eq:edge_inequality_sum}
     \sum_{uv \in E(G)} \left(\frac{(B_{uu}-1)^2-(B_{uv})^2}{B_{uu}} + \frac{(B_{vv}-1)^2-(B_{uv})^2}{B_{vv}}\right)\ge 0.
\end{equation}
To see the last equivalence, note that 
\[\deg(v) = (B^2)_{vv} = (B_{vv})^2 + \sum_{uv \in E(G)} (B_{uv})^2 + \sum_{u\in S(v)}(B_{uv})^2.\]

Clearly, \eqref{eq:edge_inequality_sum} holds by Claim \ref{claim:B_edge_inequality}. This completes the proof of the claim.
\end{proof}

Combining Claims \ref{claim:vertex_deleted_trace}, \ref{claim:B_edge_inequality} and \ref{claim:trace_sum}, we get the desired inequality \eqref{eq:deletion_inequality}.
\end{proof}

\section{Proof of Theorem \ref{thm:main}}
\label{sec:main_proof}

\begin{proof}[Proof of Theorem \ref{thm:main}]
We proceed by induction on $n$. If $n = 0$, the assertion is vacuously true. The assertion is immediate for $n=1$; so let $n\ge 2$. Observe that for every
$v\in V(G)$, adding $v$ to any independent set of $G-N[v]$ produces an independent set in $G$. Therefore, for any $v\in V(G)$,
\[ \alpha\left(G-N[v]\right)\le\alpha(G)-1.\]
The graph $G-N[v]$ has order $n-1-\deg(v)$, so the induction hypothesis implies
\[
\begin{aligned}
    \mathcal{E}\bigl(G-N[v]\bigr)
    &\ge 2\left(n-1-\deg(v)-\alpha\bigl(G-N[v]\bigr)\right)\\
    &\ge 2 \left(n-\alpha(G)-\deg(v)\right).
\end{aligned}
\]
Summing this inequality over all $v\in V(G)$, and using $\sum_v\deg(v)=2m$, we
obtain
\begin{equation}\label{eq:induction-sum}
    \sum_{v\in V(G)}\mathcal{E}\bigl(G-N[v]\bigr)
    \ge 2\left(n(n-\alpha(G))-2m\right).
\end{equation}
Lemma~\ref{lemma:deletion} and \eqref{eq:induction-sum} now imply
\[
    n\, \mathcal{E}(G)
    \ge 4m+\sum_{v\in V(G)}\mathcal{E}\bigl(G-N[v]\bigr)
    \ge 2n(n-\alpha(G)).
\]
Dividing by $n$ gives the desired inequality 
\[\mathcal{E}(G)\ge 2\left(n-\alpha(G)\right).\qedhere\]
\end{proof}

\section*{AI statement}
We acknowledge the use of AI tools during the ideation phase. We declare that the text is not AI-generated.

\bibliographystyle{plain}
\bibliography{references}

\vspace{0.3cm}
\noindent Hitesh Kumar, Email: {\tt hitesh.kumar.math@gmail.com}, {\tt hitesh\_kumar@sfu.ca}\\
\textsc{Department of Mathematics, Simon Fraser University, Burnaby, BC \ V5A1S6, Canada}\\[1pt]

\noindent Shivaramakrishna Pragada,
Email: {\tt shivaramkratos@gmail.com, shivaramakrishna\_pragada@sfu.ca}\\
\textsc{Department of Mathematics, Simon Fraser University, Burnaby, BC \ V5A1S6, Canada}\\[1pt]

\end{document}